\documentclass{siamltex}
\usepackage{amsmath}
\usepackage{amssymb,amsfonts}
\usepackage{verbatim}
\usepackage{graphicx,epstopdf}
\usepackage{xcolor}
\usepackage{bbm}
\def \R{\mathbb{R}} 

\def \N{\mathbb{N}}

\def \N0{\mathbb{N}_0}

\def \a{\alpha}

\def \W{\Omega}
\def \phi{\varphi}
\def \ge{\gamma_e}

\def \1{\mathbbm{1}}

\def \<{\left<}
\def \>{\right>}

\def \mQ{\mathcal{Q}}
\def \Th{\mathcal{T}_h}

\def \EI{\mathcal{E}_h^I}
\def \EB{\mathcal{E}_h^B}

\newcommand{\E}{\mathcal{E}_h}
\newcommand{\Eh}{\mathcal{E}_h}
\def \mJ{\mathcal{J}}
\def \mJh{\mathcal{J}_h}

\def \wto{\rightharpoonup}

\def \dv{\cdot\nu_{e}}
\def \grad{\nabla}
\def \lss{\lesssim}

\def \dx[#1]{\ensuremath{\operatorname{d}\!{#1}}}

\let\div\undefined
\DeclareMathOperator*{\div}{div}

\DeclareMathOperator*{\argmin}{arg\,min}

\DeclareMathOperator*{\sgn}{sgn}

\DeclareMathOperator*{\BV}{BV}

\newcommand{\Ome}{\Omega}

\newcommand{\nab}{\nabla}

\newtheorem{defn}{Definition}

\graphicspath{ {figures/} }

\begin{document}

\title{A Discontinuous Ritz Method for a Class of Calculus of Variations Problems}

\author{Xiaobing Feng\thanks{Department of Mathematics, The University of Tennessee,  Knoxville, TN 37996, U.S.A.
		(xfeng@math.utk.edu)}
\and 
Stefan Schnake\thanks{Department of Mathematics, The University of Tennessee,  Knoxville, TN 37996, U.S.A.
		 {\em Current Address}: Department of Mathematics, University of Oklahoma, Norman, OK 73019, U.S.A.
		 (sschnake@ou.edu) }
		}

\maketitle

\begin{abstract}
This paper develops an analogue (or counterpart) to discontinuous Galerkin (DG) methods for
approximating a general class of calculus of variations problems.  The proposed method, called 
the discontinuous Ritz (DR) method, constructs a numerical solution by minimizing a discrete energy over 
DG function spaces.  The discrete energy includes standard penalization terms as 
well as the DG finite element (DG-FE) numerical derivatives developed recently by Feng, Lewis, and Neilan 
in \cite{Feng2013}.  It is proved that the proposed DR method converges and that the DG-FE numerical derivatives
exhibit a compactness  property which is desirable and crucial for applying the proposed  DR  method  
to problems with more complex energy functionals. Numerical tests are provided on the classical $p$-Laplace problem to gauge the 
performance of the proposed DR method.   
\end{abstract}



\pagestyle{myheadings}
\thispagestyle{plain}
\markboth{XIAOBING FENG AND STEFAN SCHNAKE}{DG METHODS FOR VARITATIONAL PROBLEMS}

\section{Introduction} \label{section1}

In this paper we develop a numerical method using totally discontinuous piecewise polynomial functions for 
approximating solutions to the following problem from the calculus of variations: Find $u\in W^{1,p}_g(\W)$ such that
\begin{align}\label{dgcalc:minproblem}
\mJ(u)\leq \mJ(v)\qquad\forall v\in W_g^{1,p}(\W),
\end{align}
where
\begin{align}\label{dgcalc:eqnenergy}
\mJ(v) = \int_{\W} f(\grad v,v,x) \dx[x]
\end{align}
is the energy functional, $f:\R^d\times\R\times\W\to\R_+$ is called the energy density, $\W\subset\R^d$ is an open bounded domain, and 
$$W^{1,p}_g(\W) := \{v\in W^{1,p}(\W) : u = g\text{ on }\partial\W\}.$$  
If such a $u$ exists, it is called a minimizer of $\mJ$ over $W_g^{1,p}(\W)$ and is written as
\begin{align}\label{dgcalc:argmin}
u\in\argmin_{v\in W^{1,p}_g(\W)}\mJ(v).
\end{align}

Although the calculus of variations is an old field  in mathematics, its growth and boundary have kept expanding because new applications arising from physics, differential geometry, image processing, materials science, and optimal control (just to name a few). Those problems are often formulated as calculus of variations problems, among them are the Brachistochrone problem \cite{Dacorogna2015b}, the minimal surface problem \cite{Evans2010}, and the Erickson energy for nematic liquid crystals \cite{Nochetto2017}. 

Numerically solving those problems means to approximate the exact minimizer $u$ of $\mJ$ over $W_g^{1,p}(\W)$ via a numerical approximation $u_h$.  As expected, there are many methods for 
constructing an approximate  solution $u_h$. The existing numerical methods can be divided into two categories: the indirect approach and the direct approach. The indirect approach is based on the fact that the minimizer $u$ must satisfy, in some sense, the following Euler-Lagrange equation:
\begin{align} \label{eqn:eulerlagrange}
\sum_{i=1}^{d}\frac{\partial}{\partial x_i}\left(f_{\xi_i}(\grad u,u,x)\right) = f_u(\grad u,u,x) \qquad\forall x\in\W.
\end{align}
As equation \eqref{eqn:eulerlagrange} is a second order PDE in divergence (or conservative) form, it can be discretized using a variety of methods such as finite difference, finite element, discontinuous Galerkin and 
spectral method for constructing an approximate solution $u_h$.  This indirect approach is often the preferred approach because of the wealthy amount of numerical methods available for discretizing PDEs. However, this approach does have two drawbacks.  First, the Euler-Lagrange equation is only a necessary condition for a minimizer and it may not be a sufficient one.  More information must be known about $\mJ$ in order to determine if the solution of the Euler-Lagrange equation indeed globally minimizes $\mJ$.  Second, a discretization of the PDE may lose some important properties of the original energy functional, such as conservation or dissipation laws.   
On the other hand, the direct approach seeks an approximate solution $u_h$ by first constructing a discrete energy functional $\mJh$ and then setting 
\begin{align}\label{dgcalc:eqn:approx}
u_h \in \argmin_{v_h\in X_h}\mJh(v_h),
\end{align}
where $X_h$ is a finite-dimensional space which approximates $W_g^{1,p}(\W)$.  Since problem \eqref{dgcalc:eqn:approx} is equivalent to a minimization problem in $\R^N$, a variety of algorithms (or solvers)  can be employed to compute $u_h$.  For example, we may minimize $\mJ_h$ by using a quasi-Newton algorithm or by first deriving the (discrete) Euler-Lagrange equation to $\mJ_h$ and then solving for $u_h$.    
The key issue of this approach is how to construct a ``good'' discrete energy functional $\mJ_h$ which can ensure the convergence of $u_h$ to $u$.  One important advantage of the direct approach is that a ``good" discrete energy functional $\mJ_h$ will automatically  preserve key properties of the original energy functional $\mJ$. For example, the discrete variational derivative method by Furihata and Matsuo for 
the KdV equation, nonlinear Schr\"odinger equations, and the Cahn-Hillard equation \cite{Furihata2010}; 
the Variational DGFEM method by Buffa and Ortner \cite{Buffa2009} for calculus of variations  problems, and the finite element method  by Nochetto {\em et al.} \cite{Nochetto2017}  for nematic liquid crystals all have such a trait.   

Our goal in this paper is to develop a discontinuous Ritz (DR) framework for a class of variational 
problems described by \eqref{dgcalc:minproblem}.  Our numerical method belongs to the direct approach and takes $X_h=V_h$ - the discontinuous Galerkin (DG) space consisting of totally discontinuous piecewise polynomial functions on a mesh $\Th$ of $\W$.  We call our method a \textit{discontinuous Ritz} method because it directly approximates problem \eqref{dgcalc:minproblem}. In the special case when 
\[
\mJ(v) = \frac12 a(v,v) - F(v), 
\]
and $a(\cdot,\cdot)$ is a symmetric and coercive bilinear form,  problem \eqref{dgcalc:minproblem} is 
known as the Ritz formulation of the following Galerkin (or weak) formulation: find $u\in V$ (which is 
assumed to be a Hilbert space) such that
\[
	a(u,v) = F(v) \qquad\forall v\in V.
\]
As mentioned earlier,  the key issue we face is to construct a ``good" discrete energy functional 
$\mJ_h$.  Since DG functions are discontinuous across element edges, two roadblocks arise when 
creating a discrete energy functional $\mJ_h$ that makes sense on $V_h$.  First, $\mJ_h$ must weakly enforce  continuity and the Dirichlet boundary data. The standard way to cope with this issue in the DG framework
is to use interior penalty terms, and indeed including interior penalty terms in $\mJ_h$  is sufficient to
 obtain these properties in the limit as $h\to 0$.  Second and more importantly, these discontinuities also make DG functions not globally differentiable in general, and one has to determine how to approximate
 the gradient operator $\nabla$ in the energy functional $\mJ$.   An obvious choice is to approximate it by
 a piecewisely defined gradient operator over $\Th$. However,  such a naive choice of a discrete gradient 
 may lead to divergent numerical method \cite{Buffa2009}.  To overcome this difficulty, our idea is to use 
 the newly developed discontinuous Galerkin finite element (DG-FE) numerical derivatives (and gradient)  by 
 Feng, Lewis, and Neilan in \cite{Feng2013} as our discrete derivatives (and gradient). The bulk of this
 paper will devote to demonstrating the discrete energy functional so-constructed is a ``good" one, in the
 sense that the resulting discontinuous Ritz method  converges for a general class of energy 
functionals $\mJ$.  On the other hand,  no error estimate (or rates of convergence)  will be provided for the general framework,  such a result may only be feasible for specific problems and will be reported in a future work. 
 
The rest of the paper is organized as follows.  In Section \ref{dgcalc:notation}, we give the notation used for the paper as well as the assumptions on the density function $f$ in order to guarantee the well-posedness of problem \eqref{dgcalc:minproblem} and the convergence of the proposed discontinuous Ritz method.  In Section \ref{dgcalc:derivatives}, we illustrate the need for a proper discretization of the gradient operator 
by showing some failed choices of discrete gradients tested on Poisson problem with homogeneous Dirichlet boundary data.  In Section \ref{dgcalc:formulation}, we present the definition of the DG-FE numerical derivatives, the motivation for using it, and then define our discontinuous Ritz method.  Section \ref{dgcalc:analysis} is devoted to the convergence analysis of the proposed DR method.  We prove that the proposed DR method and the variational DGFEM by Buffa and Ortner \cite{Buffa2009} are actually equivalent schemes and verify the convergence of the proposed DR method for a class of densities $f$.  In addition, we present a compactness result using our DG-FE numerical gradient, which is of independent interests.  In Section \ref{dgcalc:examples}, we show a few numerical tests for the proposed DR  method on the $p$-Laplace problem. 

\section{Preliminaries}\label{dgcalc:notation}

\subsection{Notation}

Let $\W$ be a bounded polygonal domain in $\R^d$ ($d=1,2,3$).  For $1\leq p<\infty$, let $L^p(\W)$ and $W^{1,p}(\W)$ 
denote the usual $L^p$ space and Sobolev space on $\W$ with their standard norms.  
$(\cdot,\cdot)$ stands for the standard $L^2(\Ome)$ inner product. 
We use  $p^*>1$ to denote the Sobolev conjugate of $p$, that is, 
\begin{align}\label{dgcalc:eqnsobolevcong}
 p^* = \begin{cases} \frac{dp}{d-p} &\text{ if }p<d, \\
                     \infty         &\text{ if }p\geq d,
         \end{cases}	
\end{align}
and $q^*>1$ such that
\begin{align}\label{dgcalc:eqn:edgesobolevcong}
 q^* = \begin{cases} \frac{(d-1)p}{d-p} &\text{ if }p<d, \\
                     \infty         &\text{ if }p\geq d.
         \end{cases}	
\end{align}

Let $\Th$ be a quasi-uniform and shape regular mesh of $\W$, and let $\EI,\EB$ be the interior and boundary edges of $\Th$, and $\Eh = \EI\cup\EB$.  For any $e\in\Eh$, let $\ge>0$, the penalty parameter, be a constant on $e$ and denote $\gamma^*=\min_{e\in\EI}\ge$.

For any interior edge/face $e = \partial T^+ \cap \partial T^-\in \EI$, we define the jump and average of a scalar or
vector valued function $v$ as
\begin{align*}
[v]\big|_{e} := v^+ - v^-, \qquad
\{v\}\big|_{e} := \frac{1}{2}\left(v^+ + v^-\right),
\end{align*}
where $v^{\pm}=v|_{T^\pm}$.  On a boundary edge/face $e\in \EB$
with $e = \partial T^+\cap \partial \Omega$, we set $[v]\big|_{e} =\{v\}\big|_e =  v^+$.
For any $e\in \EI$ we use $\nu_e$ to denote the unit outward normal vector pointing in the direction of the
element with the smaller global index.  For $e\in \EB$ we set $\nu_e$ to be the outward normal to $\partial\Omega$
restricted to $e$.

We also define the broken Sobolev space 
\[
	W^{1,p}(\Th):=\prod_{T\in\Th} W^{1,p}(T)
\]
endowed with the following semi-norm and norm:
\begin{align*}
	| v |_{W^{1,p}(\Th)} &= \| \grad v \|_{L^p(\Th)} + \bigg(\sum_{e\in\EI}
	\int_e \ge h_e^{1-p} \bigl|[v] \bigr|^p\dx[S]\bigg)^{1/p}, \\
	\| v \|_{W^{1,p}(\Th)} &= | v |_{W^{1,p}(\Th)} + \bigg(\sum_{e\in\EB}\int_e \ge h_e^{1-p} |v-g|^p\dx[S]\bigg)^{1/p},
\end{align*}
where 
\[
\| \grad v \|_{L^p(\Th)} := \Bigl(\, \sum_{T\in\Th} \|\nab v\|_{L^p(T)}^p \,\Bigr)^{\frac1p}. 
\]
We define the standard discontinuous Galerkin space $V_h$ by
\[
V_h = V_h^k = \Bigl\{v_h\in L^2(\Omega);\, v_h\big|_T\in \mathbb{P}_k(T) \quad\forall T\in\Th \Bigr\},
\]
where $k\geq 0$ denotes the polynomial degree. Obviously, we have $V_h\subset W^{1,p}(\Th)$. 

\subsection{Well-posedness of calculus of variations problems}

As previously mentioned, the calculus of variations is an old field in mathematics, which gives us a solid well-posedness theory for problem \eqref{dgcalc:minproblem} with a general density function
$f$. To be precise for the remaining presentation, we shall only consider the following class of 
density functions $f$: 
\begin{enumerate}
\item[(1)] $f$ is a Carath\'{e}dory function, that is, 
\begin{itemize}
\item[(a)] $x\to f(\xi,v,x)$ is measurable for every $(\xi,v)\in \R^d\times\R$, and
\item[(b)] $(\xi,v)\to f(\xi,v,x)$ is continuous for every $x\in\W$.
\end{itemize}
\item[(2)] $\xi\to f(\xi,v,x)$ is convex for every $(v,x)\in \R\times\W$.  
\item[(3)] For fixed $1<p<\infty$, there exists constants $\a_0,\a_1>0$, $a_0,a_1\in L^1(\W)$, and $r$ and $q$ satisfying $r<p$ and $r\leq q < p^*$ such that the following growth condition holds:
\begin{align*}
\a_0\bigl( |\xi|^p-|v|^r+a_0(x) \bigr)\leq f(\xi,v,x) \leq \a_1 \bigl( |\xi|^p+|v|^q+a_1(x) \bigr).
\end{align*}
\end{enumerate}

Under above three assumptions, the direct method of calculus of variations 
(see \cite{Dacorogna2007}) shows that there exists a $u\in W_g^{1,p}(\W)$ satisfying \eqref{dgcalc:minproblem}. Moreover, if the map $(\xi,v)\to f(\xi,v,x)$ is strictly convex for a.e.~$x\in\W$, then the minimizer $u$ is unique. We refer the interested reader to 
\cite{Dacorogna2007} for the detailed proofs of these results. 

We also note that the above structure assumptions exclude the case $p=1$ and $p=\infty$. Since the spaces
$W^{1,1}$ and $W^{1,\infty}$ are non-reflexive, these two cases are expected to be difficult to deal with and
must be  considered separately.  On the other hand, we would like to mention that Stamm and Wihler in \cite{Stamm} developed a DG method for the total variation energy, which is a special problem for the case $p=1$. 
They directly discretize the $TV$-energy as 
\[
\mJ_h(v_h) = \frac{\alpha h}{2}\int_\W \sqrt{|\grad_h v_h|^2 + \beta} \, dx  + \frac{1}{2}\|f-v_h\|_{L^2(\W)}^2, 
\]
where $f$ is the given noisy image function and $\grad_h$ is the DG-FE numerical derivative introduced in Section \ref{dgcalc:formulation}.  
As we will see later, this work is in the same spirit as ours, and the numerical tests given in \cite{Stamm} are quite promising. 

\section{The choice of discrete derivatives}\label{dgcalc:derivatives}

Since functions in the discontinuous Galerkin space $V_h$ are discontinuous across element edges, 
the energy functional $\mJ$ is not defined on $V_h$.  To extend its domain to $V_h$, we define the 
following discrete energy functional:
\begin{align}
\mJ_h^*(v_h) &= \sum_{T\in\Th}\int_T f(\grad^* v_h,v_h,x) \dx[x] + \sum_{e\in\EI}\int_e \ge h_e^{1-p}\bigl|[v_h] \bigr|^p\dx[S] \\
&\quad+ \sum_{e\in\EB}\int_e \ge h_e^{1-p}|v_h-g|^p\dx[S], \nonumber
\end{align}
we note that the last two terms, which are called penalty terms, are used to weakly enforce the continuity and the Dirichlet boundary data. Here $\grad^*$ denotes an under-determined discrete gradient defined on $V_h$ or more generally on $W^{1,p}(\Th)$.  

Below we shall show that the construction of this discrete gradient $\grad^*$ is crucial to the 
convergence of the numerical method, even when penalty terms are added,
it must be defined judiciously to ensure the convergence.  
We note that the discontinuous nature of DG functions allows us to have flexibility in choosing 
the discrete gradient and to take into consideration of the properties such as simplicity and 
ease of implementation.  

The simplest approach is to define $\grad^*$ to be the piecewise gradient, that is,
$(\grad^* v)|_T = \grad(v|_T)$ for any $T\in\Th$. Obviously, such a discrete gradient is very 
easy and cheap to compute.  This gives us the following discrete energy functional:
\begin{align}\label{dgcalc:eqn:piecewiseenergy}
\mJ_h^{pw}(v_h) &= \sum_{T\in\Th}\int_T f(\grad^* v_h,v_h,x) \dx[x] + \sum_{e\in\EI}\int_e \ge h_e^{1-p}|[v_h]|^p\dx[S] \\
&\quad+ \sum_{e\in\EB}\int_e \ge h_e^{1-p}|v_h-g|^p\dx[S]. \nonumber
\end{align}
However, it is mentioned in \cite{Buffa2009} that the above approach does not always give 
a convergent scheme.  Indeed, this is true even for nice $f$. To see why it is so, let $p=2$, $g=0$, and $f(\xi,v,x) = \frac{1}{2}|\xi|^2-F(x)v$,  it is easy to check that the Euler-Lagrange equation of \eqref{dgcalc:minproblem} is the following Poisson problem:
\begin{subequations}\label{dgcalc:eqn:poiprob}
\begin{alignat}{2}
-\Delta u &= F &&\quad \text{ in }\W, \\
u &= 0 &&\quad \text{ on }\partial\W. 
\end{alignat}
\end{subequations}

Requiring the G\^ateaux derivative of $\mJ_h^{pw}$ to vanish at a potential minimizer $u_h\in V_h$, 
that is, for every $v_h\in V_h$
\[
	\frac{\text{d}}{\text{d}t} \mJ_h^{pw}(u_h+tv_h)\bigg|_{t=0} = 0 \qquad\forall v_h\in V_h,
\]
we arrive at the following problem: find $u_h\in V_h$ such that
\begin{align}\label{dg_problem}
a_h^{pw}(u_h,v_h) = (F,v_h)\qquad\forall v_h\in V_h,
\end{align}
where
\begin{align*}
a_h^{pw}(u_h,v_h) &:= \sum_{T\in\Th}\int_T \grad u_h\cdot\grad v_h\dx[x] + \sum_{e\in\EI}\int_e \frac{\ge}{h_e}[u_h][v_h]\dx[S] \\
&\quad+ \sum_{e\in\EB}\int_e \frac{\ge}{h_e} u_hv_h\dx[S]. \nonumber
\end{align*}
It is easy to verify that the bilinear form $a_h^{pw}(\cdot,\cdot)$ is coercive and continuous on $V_h$ 
for any $\ge>0$, which immediately implies the existence and uniqueness of a solution $u_h$ to 
problem \eqref{dg_problem}. However, it is not hard to prove that scheme \eqref{dg_problem}
is not consistent to the PDE problem \eqref{dgcalc:eqn:poiprob},  because
if $u$ is the weak solution to \eqref{dgcalc:eqn:poiprob}, there is a $v_h\in V_h$ such that
\[
	a_h^{pw}(u,v_h) \neq (F,v_h).
\]
Instead we have
\[
	a_h^{pw}(u,v_h) = (F,v_h) + \sum_{e\in\EI}\int_e \{\grad u\dv\}[v_h]\dx[S] \qquad \forall v_h\in V_h.
\]
We emphasize that the penalty terms are not the cause for the inconsistency, since the regularity and 
boundary data of $u$ forces them to vanish.  It is in fact the discretization of the gradient that 
causes the inconsistency. The inconsistency in this example, being $\mathcal{O}(\ge^{-1})$, leads to a non-convergent method.  To show this, we let $d=2$, $\W=(-1/2,1/2)^2$ and choose $F$ such that the 
solution $u(x,y) = (1/4-x^2)(1/4-y^2)$.  Table \ref{table:inconsist} shows the piecewise $H^1$ errors 
and rates for varying values of $\ge$. As we can see, the method is not converging to $u$ as $h\to 0$.  

\begin{table}[h]                                             
\centering                 
\caption{The piecewise $H^1$ errors and rates of convergence with various $\ge$ for the piecewise gradient discretization.  Here the polynomial degree $k=2$ is used in the test.}                                      
\begin{tabular}{|c|c|c|c|c|c|c|}                     
\hline
&\multicolumn{2}{|c|}{$\ge = 10$} & \multicolumn{2}{|c|}{$\ge = 100$} & \multicolumn{2}{|c|}{$\ge = 1000$} \\
\hline
$1/h$ & $H^1$ Error & Rate & $H^1$ Error & Rate & $H^1$ Error & Rate \\
\hline                                                         
2 & 1.69e-02 & - & 1.16e-02 & - & 1.53e-02 & - \\     
\hline                                                         
4 & 1.18e-02 & 0.52 & 2.75e-03 & 2.08 & 3.59e-03 & 2.09 \\     
\hline                                                         
8 & 1.11e-02 & 0.09 & 1.31e-03 & 1.07 & 8.60e-04 & 2.06 \\     
\hline                                                         
16 & 1.11e-02 & -0.00 & 1.31e-03 & -0.00 & 2.21e-04 & 1.96 \\  
\hline                                                         
32 & 1.12e-02 & -0.01 & 1.34e-03 & -0.04 & 1.32e-04 & 0.75 \\  
\hline                                                         
64 & 1.12e-02 & -0.01 & 1.36e-03 & -0.01 & 1.35e-04 & -0.04 \\ 
\hline                                                         
128 & 1.12e-02 & -0.00 & 1.36e-03 & -0.00 & 1.38e-04 & -0.03 \\
\hline                                                         
256 & 1.12e-02 & -0.00 & 1.36e-03 & -0.00 & 1.39e-04 & -0.01 \\
\hline                                                         
\end{tabular}                                                                                       
\label{table:inconsist}                                     
\end{table} 

We also note that the piecewise gradient discretization has the ability to produce a consistent 
scheme if we include additional terms to the discrete energy functional. For example, for the Poisson problem, 
the standard symmetric interior penalty DG bilinear form is 
\begin{align*}
a_h^{SIPDG}(u_h,v_h) &= \sum_{T\in\Th}\int_T \grad_h u_h\cdot\grad_h v_h\dx[x] \\
&\quad - \sum_{e\in\EI}\int_e [u_h]\{\grad v_h\dv\} \dx[S] - \sum_{e\in\EI}\int_e [v_h]\{\grad u_h\dv\}\dx[S] \\
&\quad+ \sum_{e\in\EI}\int_e \frac{\ge}{h_e}[u_h][v_h]\dx[S] + \sum_{e\in\EB}\int_e \frac{\ge}{h_e} u_hv_h\dx[S].
\end{align*}
It can be shown that $a_h^{SIPDG}(\cdot,\cdot)$, being symmetric, is induced by the following discrete energy functional (cf. \cite{FengLi2015}):
\begin{align*}
\mJ_h^{SIPDG}(v_h) &= \sum_{T\in\Th}\frac{1}{2}\int_T |\grad v_h|^2 \dx[x] - \sum_{e\in\EI}\int_e[v_h]\{\grad v_h\dv\}\dx[S] \\
&\quad+ \sum_{e\in\EI}\frac{1}{2}\int_e \frac{\ge}{h_e}|[v_h]|^2\dx[S] + \sum_{e\in\EB}\frac{1}{2}\int_e \frac{\ge}{h_e}|v_h-g|^2\dx[S]. \nonumber
\end{align*} 
However, this energy is specific to the Poisson problem and cannot be extended to the 
class of density functions $f$ discussed in this paper.  

While defining the numerical gradient as the piecewise gradient does not give a convergent method, 
there are examples of successful discrete gradients. In \cite{Buffa2009}, Buffa and Ortner introduced a \textit{variational DGFEM}.  This method provided a consistent discretization of the gradient that produces a convergent method for a class of convex and coercive densities.  Their discrete gradient is defined 
using the piecewise gradient with help of the following lifting operator $R:W^{1,p}(\Th) \to [V_h]^d$:
\begin{align}\label{dgcalc-eqn-lifting}
\int_\W R(v)\cdot\phi_h = -\sum_{e\in\EI}\int_e [v]\{\phi_h\dv\}\dx[S] \quad\forall \phi_h\in [V_h]^d.
\end{align}
The motivation of using this lifting operator arises from accounting for the contribution of the 
jumps of a discontinuous function to its distributional derivative. They then defined the following 
discrete energy functional:
\begin{align}\label{dgcalc:eqn:buffaortner}
\mJ_h^{BO}(v_h) &= \sum_{T\in\Th}\int_T f(\grad v_h + R(v_h),v_h,x) \dx[x] \\
&\quad+ \sum_{e\in\EI}\int_e \ge h_e^{1-p}|[v_h]|^p\dx[S] + \sum_{e\in\EB}\int_e \ge h_e^{1-p}|v_h-g|^p\dx[S] \nonumber.
\end{align}
The bilinear form induced from this energy functional for the the Poisson problem is
\begin{align*}
a_h^{BO}(u_h,v_h) &= \sum_{T\in\Th}\int_T \grad u_h\cdot\grad v_h\dx[x] + \int_\W R(u_h)\cdot R(v_h) \dx[x] \\
&\quad - \sum_{e\in\EI}\int_e [u_h]\{\grad v_h\dv\} \dx[S] - \sum_{e\in\EI}\int_e [v_h]\{\grad u_h\dv\}\dx[S] \\
&\quad+ \sum_{e\in\EI}\int_e \frac{2\ge}{h_e}[u_h][v_h]\dx[S] + \sum_{e\in\EB}\int_e \frac{2\ge}{h_e} u_hv_h\dx[S],
\end{align*}
which is continuous and coercive on $V_h$ for sufficiently large $\ge>0$.  Moreover, 
$a_h^{BO}(\cdot,\cdot)$ is consistent to the PDE problem since 
\begin{align*}
\int_\W R(u)\cdot R(v_h) \dx[x] = \sum_{e\in\EI}\int_e [u]\{\grad R(v_h)\dv\}\dx[S] = 0
\qquad \forall v_h\in V_h,
\end{align*}
which contributes to the convergence of the method for the Poisson problem. 

Furthermore, it was proved in \cite{Buffa2009} that the lifting operator ensures compactness 
of the discrete minimizers $u_h$.  Since the minimizer of $\mJ_h^{BO}$ is sought in $V_h,$ 
which is not a subspace of $W^{1,p}(\W)$, the reflexive property of $W^{1,p}(\W)$ cannot be used 
to obtain a weakly convergent subsequence. However, $V_h$ is a subset of $\BV(\W)$, the space of 
functions with bounded variations, which does have a compactness property in the weak$\ast$ topology.  
This compactness alone only shows that a subsequence $u_{h_j}$ converges to a $u\in BV(\W)$,
but Buffa and Ortner were able to prove a stronger result: if the sequence of discrete minimizers 
$u_h$ is bounded in $W^{1,p}(\Th)$, then a subsequence $u_{h_j}$ converges to $u\in W^{1,p}(\W)$.  
Moreover, there holds the weak convergence 
\begin{align*}
	\grad^* u_{h_j}+R(u_{h_j})\wto \grad^* u \text{ in } L^p(\Ome) \quad\mbox{as } h\to 0,
\end{align*}
where $\grad^* u_{h_j}$ denotes the piecewise gradient of $u_{h_j}$.  This compactness requires the lifting operator to be present in the discretization in order to pass the week limit and prove convergence of the method.   

\section{The DG-FE numerical derivatives and the discontinuous Ritz framework} \label{dgcalc:formulation}

\subsection{The DG-FE numerical derivatives}\label{dgcalc:subsect:dgderiv}

To define the DG-FE numerical derivatives, we first introduce some notation used in \cite{Feng2013}.  Let $i=1,\ldots,d$.  We define the following trace operators $\mQ_i^+,\mQ_i^-,\mQ_i$ on every $e\in\EI$:
\begin{align}\label{dgcalc:eqn:traceoper}
\mQ_i^\pm(v) = \{ v \} \pm \frac{1}{2}\sgn(\nu_e^i)[v], \qquad 
\mQ_i(v) = \frac{1}{2}\left(\mQ_i^+(v)+\mQ_i^-(v)\right), 
\end{align}
where $\nu_e^i$ denotes the $i$\textsuperscript{th} component of the normal vector $\nu_e$ to $e\in\Eh$, and 
\[
	\sgn(\xi) = \begin{cases} 1 &\text{ if } \xi\geq 0, \\ -1 &\text{ if } \xi<0. \end{cases}
\]
For $e\in\EB$, we define $\mQ_i^+v=\mQ_i^-v=\mQ_iv=v$.  
Using these trace operators, three numerical partial derivative operators corresponding the left, right, and central traces of $v$ were defined in \cite{Feng2013} as follows.
\begin{defn}\label{dgcalc:defn:dgderiv}
Let $v\in W^{1,p}(\Th)$ and $i=1,\ldots,d$.  Define the numerical partial derivative operators in the $x_i$ coordinate $\partial_{h,x_i}^+,\partial_{h,x_i}^-,\partial_{h,x_i}:W^{1,p}(\Th)\to V_h$ by
\begin{align}\label{dgcalc:eqn:numerderiv}
\int_\W \partial_{h,x_i}^\pm(v) \phi_h \dx[x] &= \sum_{e\in\Eh}\int_e \mQ_i^\pm(v)\nu_e^i[\phi_h]\dx[S] - \sum_{T\in\Th} \int_T v\,\partial_{x_i}\phi_h\dx[x] \quad\forall \phi_h\in V_h, \\
\partial_{h,x_i}(v) &= \frac{1}{2}\left(\partial_{h,x_i}^+(v) + \partial_{h,x_i}^-(v)\right).
\end{align}
\end{defn}
We call $\partial_{h,x_i}(v)$ the central numerical partial derivative in the $x_i$ coordinate.  The motivation for these numerical derivatives is to require the standard integration by parts formula to hold when tested against any discrete function $\phi_h\in V_h$.  This allows many of the properties of the
classical derivatives to hold for the numerical derivatives; among them are the product rule, chain rule, and integration by parts (cf. \cite{Feng2013}).  Because of this, a discrete energy built using the DG-FE derivatives should be consistent. In addition, the discrete gradient operators $\grad_h^+,\grad_h^-,\grad_h:W^{1,p}(\Th)\to [V_h]^d$ were also naturally defined 
in \cite{Feng2013} by 
\begin{align}\label{dgcalc:eqn:numergrad}
 \grad_h^\pm v  &= [\partial_{h,x_1}^\pm(v),\partial_{h,x_2}^\pm(v),\ldots,\partial_{h,x_d}^\pm(v)], \\
 \grad_h v &= [\partial_{h,x_1}(v),\partial_{h,x_2}(v),\ldots,\partial_{h,x_d}(v)]. \label{dgcalc:eqn:numergrad1}
\end{align}

We describe two convergent methods which were developed in \cite{Feng2013} with the help of the DG-FE derivatives.  Both methods were formulated for problem \eqref{dgcalc:eqn:poiprob}.
To introduce these methods, we first define a jump operator $j_h:W^{1,p}(\Th)\to V_h$ as follows:
\begin{align*}
\sum_{T\in\Th}\int_T j_h(v)\phi_h \dx[x] = \sum_{e\in\EI}\int_e \frac{\ge}{h_e}[v][\phi_h]\dx[S] + \sum_{e\in\EB}\int_e \frac{\ge}{h_e}v\phi_h\dx[S]\qquad\forall \phi_h\in V_h.
\end{align*}

The first method seeks a function $u_h\in V_h$ such that 
\begin{align}
\int_\W \grad_h u_h \cdot \grad_h \phi_h \dx[x] - \sum_{e\in\EB}\int_e \grad_h u_h\cdot \nu_e\phi_h\dx[S] + \int_\W j_h(u_h)\phi_h \dx[x] = \int_\W f\phi_h  \dx[x]
\end{align}
for all $\phi_h\in V_h$.  This method is equivalent to the well-known local DG method for the model problem \cite{Cockburn1998} and converges provided $\ge>0$.  

The second method, the symmetric dual-wind discontinuous Galerkin (DWDG) method \cite{Lewis2014}, is constructed from the ground up using the DG-FE gradients. The DWDG method seeks $u_h\in V_h$ such that  
\begin{align}
\frac{1}{2}\int_\W \left(\grad_h^+ u_h \cdot \grad_h^+ \phi_h + \grad_h^- u_h \cdot \grad_h^- \phi_h \right)\dx[x] + \int_\W j_h(u_h)\phi_h \dx[x] = \int_\W f\phi_h  \dx[x]
\end{align}
for all $\phi_h\in V_h$.  Note that the sided gradients $\grad_h^+$ and $\grad_h^-$, instead of the central gradient, are used in the formulation.  If $\ge>0$, then the method was proved to be well-posed and convergent. Moreover, if $\Th$ is quasi-uniform and if each element $T\in\Th$ has at most one boundary edge,
then the method is well-posed and converges provided $\ge>-C_*$ for some $h$-independent constant $C_*>0$.  Thus one could set $\ge\equiv 0$, that is, ignoring the penalty terms, and still achieve convergence.  

We also note that besides their applications in solving PDEs, a complete DG-FE numerical calculus was developed in \cite{Feng2013}, which is of independent interests as it provides an alternative approach 
for computing weak (and distributional) derivatives of non-smooth functions.  A Matlab Toolbox 
was recently developed in \cite{Schnake2017,DGCalccode} for implementation of this DG-FE numerical calculus 
in one and two dimensions. The toolbox provides a convenient software package for both teaching and 
research related to numerical derivatives. 

\subsection{Formulation of the discontinuous Ritz method}\label{dgcalc:subsect:dritzform}

With the DG-FE gradients in hand, we are ready to introduce our discontinuous Ritz (DR) method.
 
\begin{defn}\label{dgcalc:defnritz}
The discontinuous Ritz method for problem \eqref{dgcalc:eqn:poiprob} is defined by 
seeking $u_h\in V_h$ such that 
\begin{align} \label{dgcalc:eqn:ritzmini}
	u_h\in \argmin_{v_h\in V_h}\mJ_h(v_h),
\end{align}
where
\begin{align}\label{dgcalc:eqn:ritzenergy}
 \mJ_h(v) &= \int_\W f(\grad_h v,v,x)\dx[x] + \sum_{e\in\EI}\int_e \ge h_e^{1-p}|[v]|^p\dx[S] \\
 &\qquad+ \sum_{e\in\EB}\int_e \ge h_e^{1-p}|v-g|^p\dx[S], \nonumber
\end{align}
where $\nab_h$ is defined by \eqref{dgcalc:eqn:numergrad1}.
\end{defn}

To compute the numerical derivative $\partial_{h,x_i}v$, we note that the mass matrix induced by the left-hand side of \eqref{dgcalc:eqn:numerderiv} is actually a block diagonal matrix which means the computation of the derivatives can be done locally and in parallel. Moreover, when determining the DG-FE partial derivatives of a discrete function, the linearity of $\partial_{h,x_i}^\pm$ and $\partial_{h,x_i}$ allows the action of taking the DG-FE partial derivatives to be written as a matrix which can be computed off-line 
(cf. \cite{DGCalccode}).  

\section{Convergence analysis of the discontinuous Ritz method}\label{dgcalc:analysis}

Clearly, the definition of our DR method is quite simple, we simply replace the 
differential gradient operator $\nab$ by the DG-FE (central) numerical gradient $\nab_h$ in the 
the energy functional $\mJ$ to obtain our discrete energy functional $\mJ_h$.  On the other hand,
the convergence analysis of the proposed DR method is much less straightforward.  It is not clear
at the first look why the method would work.  
The goal of this section is to show the convergence.  This will be done indirectly by 
showing that the proposed DR method as defined in Definition \ref{dgcalc:defnritz} is actually 
equivalent to the variational DGFEM developed by Buffa and Ortner in \cite{Buffa2009}. Specifically,
we shall prove $\mJ_h \equiv \mJ_h^{BO}$ on $V_h$, thus giving equivalence of these two methods 
when minimizing over $V_h$, the equivalence allows us to borrow many technical results from \cite{Buffa2009}. We also present conditions to give the equivalence of $\|\grad_h v_h\|$  
and $|v_h|_{W^{1,p}(\Th)}$ as well as a compactness result for the DG-FE derivatives.  

First, we show the equivalence of $\mJ_h^{BO}$ and $\mJ_h$ on $V_h$.

\begin{lemma}\label{dgcalc:lemequiv:1}
Let $\mJ_h^{BO}$ and $\mJ_h$ be defined by \eqref{dgcalc:eqn:buffaortner} and \eqref{dgcalc:eqn:ritzenergy} respectively, then for any $v_h\in V_h$ we have $\mJ_h(v_h)=\mJ_h^{BO}(v_h)$.  
\end{lemma}

\begin{proof}
Let $v_h\in V_h$, if we can show that $\grad_h v_h = \grad v_h + R(v_h)$, where $\grad v_h$ is the piecewise gradient, then the equivalence of the two methods follows.  This property was already proved in Proposition 4.2 of \cite{Feng2013}, but below we include the whole proof for completeness.  

We first state the DG integration by parts formula:
\begin{align} \label{eqn2:8}
\sum_{T\in\Th}\int_T\tau\cdot\grad v \dx[x]&= -\sum_{T\in\Th}\int_{T}v\div\tau \dx[x] \\ \nonumber
&\quad+\sum_{e\in\EI}\int_e [\tau\dv]\{v\} \dx[S] +\sum_{e\in\E}\int_e\{\tau\dv\}[v]\dx[S],
\end{align}
which holds for any $v\in W^{1,p}(\Th)$ and $\tau \in [W^{1,p}(\Th)]^d$.

For any $v\in W^{1,p}(\Th)$, by the definition of $\grad_h v_h$ and \eqref{eqn2:8}, we have
\begin{align} \label{dgcalc:eqn:derivtolift}
\int_\W \grad_h v\cdot\phi_h &= \sum_{e\in\Eh}\int_e \{v\}[\phi_h\dv]\dx[S] - \sum_{T\in\Th}\int_T v\div\phi_h\dx[x] \\
&= -\sum_{e\in\EI}\int_e [v]\{\phi_h\dv\}\dx[S] + \sum_{T\in\Th}\int_T \grad v\cdot\phi_h\dx[x] \nonumber\\
&= \sum_{T\in\Th}\int_T (\grad v+R(v))\cdot\phi_h\dx[x] \nonumber\\
&= \int_\W (\grad v+R(v))\cdot\phi_h\dx[x].\nonumber \qquad\forall \phi_h\in[V_h]^d.
\end{align}
Thus we have
\begin{align*}
\int_{\W} \big(\grad_h v_h - (\grad v_h+R(v_h))\big)\cdot\phi_h \dx[x] = 0 \qquad\forall \phi_h\in[V_h]^d,
\end{align*}
by \eqref{dgcalc:eqn:derivtolift}.  Since $\grad_h v_h, \grad v_h, R(v_h)\in [V_h]^d$, setting $\phi_h = \grad_h v_h - (\grad v_h+R(v_h))$ we obtain $\grad_h v_h = \grad v_h+R(v_h)$ in $\W$.  Thus $\mJ_h(v_h)=\mJ_h^{BO}(v_h)$.  The proof is complete.  
\end{proof}

With the equivalence we can borrow and take advantage of the convergence result from Theorem 6.1 of \cite{Buffa2009}. 
\begin{theorem}
For $h>0$, let $u_h\in V_h$ satisfy \eqref{dgcalc:eqn:ritzmini}.  Then there exists a sequence $h_j\searrow 0$ and a function $u\in W_g^{1,p}(\W)$ such that the following hold:
\begin{subequations}
\begin{align}
	&u_{h_j} \to u  \text{ in } L^q(\W) \quad\forall q<p^* \nonumber, \\
	&\grad_{h_j}u_{h_j} \wto \grad u \text{ in } [L^p(\W)]^d \nonumber, \\
	&\mJ_{h_j}(u_{h_j}) \to \mJ(u), \nonumber \\
	&\sum_{e\in\EB}\int_e h_e^{1-p}|u_{h_j}-g|^p\dx[S] + \sum_{e\in\EI}\int_e h_e^{1-p}|[u_{h_j}]|^p\dx[S] \to 0 \nonumber
\end{align}
\end{subequations}
as $j\to\infty$.  Moreover, any accumulation point of the set $\{u_h\}_{h>0}$ is a minimizer of $\mJ$ over $W_g^{1,p}(\W)$.  If $\xi\to f(\xi,v,x)$ is strictly convex for all $(v,x)\in\R\times\W$, then we have
\[
	\|u - u_{h_j}\|_{W^{1,p}(\Th)} \to 0 \qquad\mbox{as } j\to \infty.
\]
If the minimizer $u$ is unique, then the whole sequence $\{u_h\}_{h>0}$ converges.  
\end{theorem}

The following results will be quite useful in later use of the DF-FE derivatives.  First, we sate 
conditions to guarantee equivalence of the semi-norms $\|\grad_h \cdot \|$ and $|\cdot|_{W^{1,p}(\W)}$ 
on $V_h$. To this end, we need to quote a discrete inf-sup condition from Buffa and Ortner \cite{Buffa2009}.  

\begin{lemma}[Lemma A.2 of \cite{Buffa2009}] \label{dgcalc:infsuplem}
Let $1\leq p <\infty$ and $q$ be its H\"older conjugate.  Then there exists a constant $C>0$ independent of $h$ such that
\begin{align}\label{dgcalc:eqn:infsup}
\inf_{v_h\in V_h}\sup_{\phi_h\in V_h}\frac{\int_\W v_h\phi_h}{\|v_h\|_{L^p(\W)}\|\phi_h\|_{L^q(\W)}}\geq C.
\end{align}
\end{lemma}

We first show that $\|\grad_h v  \|_{L^p(\Th)}$ can be controlled
 by $|v|_{W^{1,p}(\Th)}$ on $W^{1,p}(\Th)$.  

\begin{lemma} \label{dgcalc:lem:normequiv1}
Let $1<p<\infty$.   Then there exists a constant $C>0$ independent of $h$ such that
\begin{align}\label{dgcalc:eqn:dgderivbounded}
\|\grad_h v \|_{L^p(\Th)} \lss |v|_{W^{1,p}(\Th)}\qquad\forall v\in W^{1,p}(\Th),
\end{align}
\end{lemma}

\begin{proof}
Let $q$ be the H\"older conjugate of $p$ and let $v\in W^{1,p}(\Th)$ and $\phi_h\in [V_h]^d$.  From \eqref{dgcalc:eqn:derivtolift} we have
\begin{align*}
\int_\W \grad_h v\cdot\phi_h \dx[x] &= -\sum_{e\in\EI}\int_e [v]\{\phi_h\dv\}\dx[S] + \sum_{T\in\Th}\int_T \grad v\cdot\phi_h\dx[x] \\
&\leq \sum_{e\in\EI}\int_e h_e^{\frac{1-p}{p}}|[v]|\cdot h_e^{\frac{1}{q}}|\{\phi_h\dv\}|\dx[S] + \sum_{T\in\Th} \|\grad v\|_{L^p(T)} \|\phi_h\|_{L^q(T)} \\
&\leq \sum_{e\in\EI}\int_e \left(h_e^{1-p}|[v]|^p\right)^\frac{1}{p}\left( h_e|\{\phi_h\dv\}|^q\right)^\frac{1}{q}\dx[S] + \|\grad v\|_{L^p(\W)} \|\phi_h\|_{L^q(\W)} \\
&\leq \bigg(\sum_{e\in\EI}h_e^{1-p}\|[v]\|_{L^p(e)}^p\bigg)^\frac{1}{p}\bigg(\sum_{e\in\EI} h_e\|\{\phi_h\dv\}\|_{L^q(e)}^q\bigg)^\frac{1}{q} \\
&\qquad  + \|\grad v\|_{L^p(\W)} \|\phi_h\|_{L^q(\W)} \\
&\lss \bigg(\sum_{e\in\EI}h_e^{1-p}\|[v]\|_{L^p(e)}^p\bigg)^\frac{1}{p}\|\phi_h\|_{L^q(\W)} +\|\grad v\|_{L^p(\Th)} \|\phi_h\|_{L^q(\W)} \\
&\lss |v|_{W^{1,p}(\Th)}\|\phi_h\|_{L^q(\W)}.
\end{align*}
Since $\grad_h v\in V_h$, it follows from Lemma \ref{dgcalc:infsuplem} that
\begin{align*}
\|\grad_h v\|_{L^p(\Th)} \lss \sup_{\phi_h\in V_h}\frac{\int_\W \grad_h v\cdot\phi_h}{\|\phi_h\|_{L^q(\W)}} \lss |v|_{W^{1,p}(\Th)}.
\end{align*}
which is exactly \eqref{dgcalc:eqn:dgderivbounded}.
\end{proof}

We next show that $|v_h|_{W^{1,p}(\Th)}$ can be controlled by $\|\grad_h v_h \|_{L^p(\Th)}$ on $V_h$ for sufficiently large $\gamma^*$.  

\begin{lemma} \label{dgcalc:lem:normequiv2}
Let $1<p<\infty$.   Then there exists a constant $C,\gamma^*>0$ independent of $h$ such that for every $v_h\in V_h$
\begin{align}\label{dgcalc:eqn:dgderivbounded2}
|v_h|_{W^{1,p}(\Th)} \leq C \|\grad_h v_h \|_{L^p(\Th)} + C\bigg(\sum_{e\in\EI}\ge h_e^{1-p}\|[v_h]\|_{L^p(e)}^p\bigg)^{1/p},
\end{align}
provided that $\ge>\gamma^*$. 
\end{lemma}

\begin{proof}
Let $q$ be the H\"older conjugate of $p$ and $v_h\in V_h$.  From \eqref{dgcalc:eqn:derivtolift} we have
\begin{align}\label{dgcalc:eqn:lifting}
	\int_\W \grad_hv_h\cdot\phi_h \dx[x] = -\sum_{e\in\EI}\int_e [v_h] \{\phi_h\dv\} \dx[S] + \int_\W \grad v_h\cdot\phi_h\dx[x]
\end{align}
for every $\phi_h\in [V_h]^d$.  Let $\mathcal{P}_h(\grad v_h|\grad v_h|^{p-2})$ where $\mathcal{P}_h$ 
is the local $L^2$ projection onto $\Th$ defined by
\[
	\int_T \mathcal{P}_h(\grad v_h|\grad v_h|^{p-2})\cdot \phi_h \dx[x] = \int_T \grad v_h|\grad v_h|^{p-2}\cdot\phi_h \dx[x]
\]
for all $\phi_h\in  V_h$ and $T\in\Th$.  Choosing $\phi_h=\mathcal{P}_h(\grad v_h|\grad v_h|^{p-2})$ in \eqref{dgcalc:eqn:lifting} yields 
\begin{align}\label{dgcalc:eqn:equivalenceineq1}
\int_\W \grad_hv_h\cdot \mathcal{P}_h(\grad v_h|\grad v_h|^{p-2})\dx[x] &= -\sum_{e\in\EI}\int_e [v_h] \{\mathcal{P}_h(\grad v_h|\grad v_h|^{p-2})\dv\} \dx[S] \\
&\quad+ \int_\W \grad v_h\cdot\mathcal{P}_h(\grad v_h|\grad v_h|^{p-2})\dx[x]. \nonumber
\end{align}
By the stability of $\mathcal{P}_h$ we obtain
\begin{align}
\int_\W \grad_hv_h\cdot \mathcal{P}_h(\grad v_h|\grad v_h|^{p-2})\dx[x] &\leq \|\grad_h v_h\|_{L^p(\Th)}\|\mathcal{P}_h(\grad v_h|\grad v_h|^{p-2}) \|_{L^q(\Th)} \\
&\leq \|\grad_h v_h\|_{L^p(\Th)}\|\grad v_h|\grad v_h|^{p-2} \|_{L^q(\Th)} \nonumber \\
&\leq \|\grad_h v_h\|_{L^p(\Th)}\|\grad v_h \|_{L^p(\Th)}^{p-1}. \nonumber
\end{align}
By the standard trace and inverse inequalities for DG functions, there exists $C_1>0$ independent of $h$ such that
\begin{align}\label{dgcalc:eqn:equivalenceineq2}
\sum_{e\in\EI}\int_e [v_h] &\{\mathcal{P}_h(\grad v_h|\grad v_h|^{p-2})\dv\} \dx[S] \\
&\leq \bigg(\sum_{e\in\EI} h_e^{1-p}\|[v_h]\|_{L^p(e)}^p \bigg)^\frac1p \bigg(\sum_{e\in\EI}h_e\|\{\mathcal{P}_h(\grad v_h|\grad v_h|^{p-2})\dv\}\|_{L^q(e)}^q \dx[S]\bigg)^\frac1q \nonumber \\
&\leq C_1 \bigg(\sum_{e\in\EI} h_e^{1-p}\|[v_h]\|_{L^p(e)}^p \bigg)^\frac1p \|\mathcal{P}_h(\grad v_h|\grad v_h|^{p-2}) \|_{L^q(\Th)} \nonumber \\
&\leq C_1 \bigg(\sum_{e\in\EI} h_e^{1-p}\|[v_h]\|_{L^p(e)}^p \bigg)^\frac1p \|\grad v_h \|_{L^p(\Th)}^{p-1}. \nonumber
\end{align}
By the properties of $P_h$ we have
\begin{align}\label{dgcalc:eqn:equivalenceineq3}
\int_\W \grad v_h\cdot\mathcal{P}_h(\grad v_h|\grad v_h|^{p-2})\dx[x] = \int_\W \grad v_h\cdot\grad v_h|\grad v_h|^{p-2}\dx[x] = \|\grad v_h \|_{L^p(\Th)}^p.
\end{align}
Thus by \eqref{dgcalc:eqn:equivalenceineq1}-\eqref{dgcalc:eqn:equivalenceineq3} and dividing by $\|\grad v_h\|_{L^p(\Th)}^{p-1}$ we have
\[
\|\grad_h v_h\|_{L^p(\Th)} \geq -C_1	\bigg(\sum_{e\in\EI} h_e^{1-p}\|[v_h]\|_{L^p(e)}^p \bigg)^\frac1p  + \|\grad v_h \|_{L^p(\Th)}.
\]
Choosing $\gamma^*= C_1^p+1$ gives us the desired estimate.  The proof is complete.
\end{proof}

We can also prove a compactness result using the DG-FE numerical derivatives.  For this, we use a discrete compactness result from Buffa and Ortner \cite{Buffa2009}.

\begin{lemma}[Theorem 5.2 and Lemma 8 of \cite{Buffa2009}] \label{dgcalc:BOcompact}
For $1<p<\infty$ and $0<h<1$, let $v^h\in W^{1,p}(\Th)$ such that 
\begin{align}
\sup_{0<h<1}\big(\|v^h\|_{L^1(\W)} + |v^h |_{W^{1,p}(\Th)} \big) <\infty.
\end{align}
Then there exists a sequence $h_j\searrow 0$ and a function $v\in W^{1,p}(\W)$ such that
\begin{subequations}
\begin{alignat}{3}
	v^{h_j} &\to v  &&\text{ in } L^q(\W) &&\quad\forall\, 1\leq q<p^*,  \\
	v^{h_j} &\to v  &&\text{ in } L^q(\partial\W) &&\quad\forall\, 1<q<q^*,  \\
	\grad v^{h_j} + R(v^{h_j}) &\wto \grad v &&\text{ in } [L^p(\W)]^d,&& \label{dgcalc:BOgradconv}
\end{alignat}
\end{subequations}
where $p^*$ is the Sobolev conjugate of $p$ defined in \eqref{dgcalc:eqnsobolevcong} and $q^*$ is defined in \eqref{dgcalc:eqn:edgesobolevcong}.
\end{lemma}

We are now ready to state our compactness result, which differs from Lemma \ref{dgcalc:BOcompact} by controlling DG functions using the DG-FE numerical derivatives as well as showing their DG-FE numerical derivatives weakly converge.  

\begin{theorem}\label{dgcalc:compactness}
Let $1<p<\infty$.  There exists $\gamma^*>0$ such that for any $v_h\in V_h$ with
\begin{align}
\sup_{0<h<1}\left(\|v_h\|_{L^p(\partial \W)} + \|\grad_h v_h \|_{L^p(\Th)} +  \bigg(\sum_{e\in\EI}\ge h_e^{1-p}\|[v_h]\|_{L^p(e)}^p\bigg)^{\frac1p}\right) <\infty.
\end{align}
Then there exists a sequence $h_j\searrow 0$ and a function $v\in W^{1,p}(\W)$ such that
\begin{subequations}
\begin{alignat}{3}
	v_{h_j} &\to v  &&\text{ in } L^q(\W) &&\quad\forall\,1\leq q<p^*, \label{dgcalc:eqnconv1} \\
	v_{h_j} &\to v  &&\text{ in } L^q(\partial\W) &&\quad\forall\, 1<q<q^*, \label{dgcalc:eqnconv2} \\
	\grad_{h_j}v_{h_j} &\wto \grad v &&\text{ in } [L^p(\W)]^d,&& \label{dgcalc:eqnconv3} 
\end{alignat}
\end{subequations}
where $p^*$ is the Sobolev conjugate of $p$ defined in \eqref{dgcalc:eqnsobolevcong} and $q^*$ is defined in \eqref{dgcalc:eqn:edgesobolevcong}.
\end{theorem}

\begin{proof}
From Lemma \ref{dgcalc:lem:normequiv2}, we have
\[
	|v_h|_{W^{1,p}(\Th)} \lss \|\grad_h v_h \|_{L^p(\Th)} +  \bigg(\sum_{e\in\EI}\ge h_e^{1-p}\|[v_h]\|_{L^p(e)}^p\bigg)^{\frac1p}.
\]
which shows that $v_h$ is uniformly bounded in $W^{1,p}(\Th)$.  By the Poincar\`e-Fredrichs inequality, Theorem 10.6.12 of \cite{Sp:BS}, we have 
\[
	\| v_h \|_{L^1(\W)} \lss \| v_h \|_{L^p(\W)} \lss \|v_h\|_{L^p(\partial\W)} + |v_h|_{W^{1,p}(\Th)}.
\]
Therefore, the family $\{v_h\}$ satisfies the hypothesis of Lemma \ref{dgcalc:BOcompact}, which gives us everything in the theorem except for \eqref{dgcalc:eqnconv3}.

To show  \eqref{dgcalc:eqnconv3}, we use the ideas from the proof of Theorem 5.2 of \cite{Buffa2009}. 
 Let $\phi\in [C_c^{\infty}(\W)]^d$, if we can show
\begin{align} \label{dgcalc:eqneweakdef}
\lim_{j\to \infty}\int_\W \grad_{h_j} v_{h_j}\cdot\phi\dx[x] = \int_\W \grad v\cdot\phi \dx[x].
\end{align}
then we are done.  To the end, let $\phi_{h_j}\in [V_{h_j}]^d$, from \eqref{dgcalc:eqn:derivtolift} we have
\begin{align*}
\int_\W \grad_{h_j} v_{h_j}\cdot\phi\dx[x] &= \int_\W \grad_{h_j} v_{h_j}\cdot\phi_{h_j}\dx[x] + \int_\W \grad_{h_j} v_{h_j}\cdot(\phi-\phi_{h_j})\dx[x] \\ 
&=\int_\W \bigl(\grad v_{h_j}+R(v_{h_j}) \bigr)\cdot\phi_{h_j}\dx[x] + \int_\W \grad_{h_j} v_{h_j}\cdot(\phi-\phi_{h_j})\dx[x] \\
&=\int_\W \bigl(\grad v_{h_j}+R(v_{h_j}) \bigr)\cdot\phi\dx[x] \\
&\quad + \int_\W \bigl(\grad v_{h_j}+R(v_{h_j}) \bigr)\cdot(\phi_{h_j}-\phi)\dx[x]  \\
&\quad+ \int_\W \grad_{h_j} v_{h_j}\cdot(\phi-\phi_{h_j})\dx[x].
\end{align*}
Lemma 7 of \cite{Buffa2009} and Lemma \ref{dgcalc:lem:normequiv1} imply the uniform boundedness of $\grad v^{h_j}$, $R(v^{h_j})$, and $\grad_{h_j} v^{h_j}$ in $L^p(\W)$.  Thus, choosing $\phi_{h_j}$ to be the piecewise constant average of $\phi$ on $T\in\mathcal{T}_{h_j}$ forces the rightmost two terms to vanish as $j\to\infty$.  We then obtain \eqref{dgcalc:eqneweakdef} from \eqref{dgcalc:BOgradconv}.  The proof is complete.
\end{proof}

\section{Numerical experiments}\label{dgcalc:examples}

In the section we present some numerical tests to show the effectiveness of the proposed discontinuous
Ritz method.  Our prototypical example is the following $p-$Laplace energy:
\begin{align}\label{dgcalc:eqn:plapenergy}
\mJ^p(v) = \int_\W \Bigl( \frac{1}{p}|\grad v|^p - Fv \Bigr) \dx[x],
\end{align}
minimized over the space $W^{1,p}_g(\W)$. So the density function $f(\xi,v,x)=(1/p)|\xi|^p- F(x)v$, which  satisfies all of the assumptions in the theory provided $F\in L^q(\W)$ for some $q>p$.  Moreover, the map $(\xi,v)\to f(\xi,v,x)$ is strictly convex for a.e~$x\in\W$.  Thus there is a unique minimizer $u\in W_g^{1,p}(\W)$.  The Euler-Lagrange equation \eqref{eqn:eulerlagrange} of $\mJ^p$ yields the following $p-$Laplace problem:
\begin{subequations} \label{dgcalc:eqn:plaplace}
\begin{alignat}{2}
	-\div(|\grad u|^{p-2}\grad u) &= F  &&\qquad \text{ in } \W, \\
	u &= g  &&\qquad \text{ on }\partial\W.
\end{alignat}
\end{subequations}
Note that $p=2$ gives the standard Poisson problem; however, here $p$ can be any number such that $1<p<\infty$.  We will test cases in both one and two-dimensions, varying the value of $p$.  We compute 
the discrete solution $u_h$ by minimizing the discrete energy \eqref{dgcalc:eqn:ritzmini} with $k=1$ and using the Matlab built-in function \verb|fminunc| with the initial guess $0$ unless otherwise specified.  We also let $\ge\equiv 10$ for every test unless otherwise stated.  

\medskip
{\bf Test 1 ($d=1$, $p>2$).}
Let $p=2.5$, $d=1$, $\W=(0,1)$ and $g = x$.  Choose $F(x) = -9\sqrt{3}x^2$ so that the exact solution is $u(x)=x^3$.  Table \ref{dgcalc:table:p2.5errors} shows the errors and rates in the $L^p$ and $W^{1,p}$-norm for $u-u_h$, where $u_h\in V_h$ is the discrete minimizer of \eqref{dgcalc:eqn:ritzmini}. 
The numerical results clearly indicate that the proposed DR method is converging to the correct solution and we have optimal order convergence in the $W^{1,p}$ semi-norm, but we have sub-optimal convergence rate in the $L^p$ norm.

\begin{table}[h]  
\caption{The $L^p$ and $W^{1,p}(\Th)$ errors and rates of convergence in $h$ for the discontinuous Ritz
	 method \eqref{dgcalc:eqn:ritzmini} applied to $\mJ^{p}$ from \eqref{dgcalc:eqn:plapenergy} where $p=2.5$ and $\ge\equiv 100$.}                              
\centering                                     
\begin{tabular}{|c|c|c|c|c|c |}    
\hline                                         
$1/h$ & $\|u-u_h\|_{L^p(\W)}$ & rate & $\|\grad u-\grad_h u_h\|_{L^p(\W)}$ & rate & iterations\\              
\hline                                    
10 & 5.12e-03 & - & 1.10e-01 & - & 72\\ 
\hline                                    
20 & 3.06e-03 & 0.74 & 5.51e-02 & 0.99 & 137\\ 
\hline                                    
40 & 1.67e-03 & 0.88 & 2.76e-02 & 1.00 & 276\\ 
\hline                                    
80 & 8.74e-04 & 0.93 & 1.38e-02 & 1.00 & 555\\ 
\hline                                    
160 & 4.49e-04 & 0.96 & 6.92e-03 & 1.00 & 1104\\
\hline                                    
320 & 2.28e-04 & 0.98 & 3.46e-03 & 1.00 & 2123\\
\hline    
\end{tabular}                                  
\label{dgcalc:table:p2.5errors}                     
\end{table}

\medskip
{\bf Test 2 ($d=1$, $p<2$). }
Let $p=1.5$, $d=1$, $\W=(0,1)$ and $g = 0$.  Choose $F(x)$ such that the exact solution is $u(x)=\sin(\pi x)$.  Note that
\[
w:=|\grad u|^{p-2}\grad u = \frac{\sqrt{\pi}\cos(\pi x)}{\sqrt{|\cos(\pi x)|}}
\]
is not classically differentiable since $\cos(\pi x)$ is both positive and negative on $(0,1)$, but $w\in W^{1,q}(\W)$ for all $1<q<2$ with $\grad w$ having a discontinuity at $x=0.5$.  Table \ref{dgcalc:table:p1.5errors} shows the $L^p$ and $W^{1,p}$ errors and rates of convergence for the DR  method.  We see that the rates of convergence are suboptimal for both the $L^p$ and $W^{1,p}$ errors.  This is most likely due to the degeneracy of the PDE since largest error occurs at $x=0.5$ where $w$ is 0.  This claim is supported by Figure \ref{dgcalc:table:p1.5fig}.

\begin{table}[h]
\caption{The $L^p$ and $W^{1,p}(\Th)$ errors and rates of convergence in $h$ for the discontinuous Ritz method \eqref{dgcalc:eqn:ritzmini} applied to $\mJ^{p}$ from \eqref{dgcalc:eqn:plapenergy} where $p=1.5$}                            
\centering                                     
\begin{tabular}{|c|c|c|c|c|c|}    
\hline                                         
$1/h$ & $\|u-u_h\|_{L^p(\W)}$ & rate & $\|\grad u-\grad_h u_h\|_{L^p(\W)}$ & rate & iterations\\              
\hline                                    
10 & 8.50e-02 & - & 3.19e-01 & - & 79\\ 
\hline                                    
20 & 5.77e-02 & 0.56 & 2.06e-01 & 0.63 & 142\\ 
\hline                                    
40 & 4.03e-02 & 0.52 & 1.38e-01 & 0.57 & 242\\ 
\hline                                    
80 & 2.85e-02 & 0.50 & 9.56e-02 & 0.53 & 415\\ 
\hline                                    
160 & 2.02e-02 & 0.50 & 6.69e-02 & 0.51 & 713\\
\hline                                    
320 & 1.43e-02 & 0.50 & 4.72e-02 & 0.51 & 1244\\
\hline    
\end{tabular}                                  
\label{dgcalc:table:p1.5errors}                     
\end{table}

\begin{figure}
\includegraphics[width=1.0\textwidth]{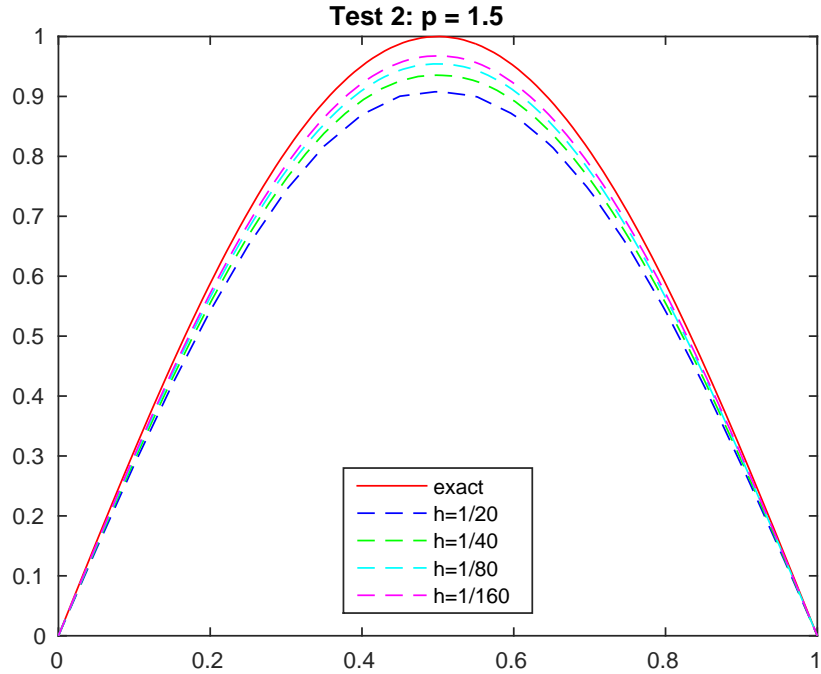}
\caption{The plots of $u$ and $u_h$ where $u$ is the exact minimizer for $\mJ^{p}(\cdot)$ from \eqref{dgcalc:eqn:plapenergy} with $p=1.5$ and $u_h$ is the discrete minimizer from \eqref{dgcalc:eqn:ritzmini}. Here $h=1/20,1/40,1/80,1/160$.}
\label{dgcalc:table:p1.5fig}  
\end{figure}

\medskip
{\bf Test 3 (Unknown solution case).}
Let $p=8.3$, $d=1$, $\W=(0,1)$ and $g = x/2$.  We choose $F(x) = \frac{2000x}{(100x^2+1)^2}$.  Since we do not know the exact solution to this problem, we choose $u^{FE}$ such that
\[
	u^{FE} = \argmin_{v\in S}\mJ(v_h)
\] 
where $S\subset W^{1,p}(\W)$ is the $C^0$ conforming Lagrange finite element space with $k=1$ and $h=1/640$. Table \ref{dgcalc:table:p8.3errors} shows the errors and rates in the $L^p$ and $W^{1,p}$-norm for $u^{FE}-u_h$, where $u_h\in V_h$ is the discrete minimizer of \eqref{dgcalc:eqn:ritzmini}.  For this test we set an initial guess of $u_0 = x/2$.   We see that the method is converging with a suboptimal rate of convergence in the $L^p$-norm.

\begin{table}[h]
\caption{The $L^p$ and $W^{1,p}$ errors and rates of convergence in $h$ for the discontinuous Ritz method \eqref{dgcalc:eqn:ritzmini} applied to $\mJ^{p}$ from \eqref{dgcalc:eqn:plapenergy} where $p=8.3$}                            
\centering                                     
\begin{tabular}{|c|c|c|c|c|c|}    
\hline                                         
$1/h$ & $\|u^{FE}-u_h\|_{L^p(\W)}$ & rate & $\|\grad u^{FE}-\grad_h u_h\|_{L^p(\W)}$ & rate & iterations\\              
\hline                                    
10 & 1.95e-02 & - & 6.43e-01 & - & 95\\ 
\hline                                    
20 & 9.28e-03 & 1.08 & 4.23e-01 & 0.61 & 255\\ 
\hline                                    
40 & 4.46e-03 & 1.06 & 6.33e-02 & 2.74 & 1026\\ 
\hline                                    
80 & 2.21e-03 & 1.01 & 2.76e-01 & -2.12 & 1618\\ 
\hline                                    
160 & 1.10e-03 & 1.01 & 2.27e-01 & 0.27 & 2763\\
\hline                                    
320 & 5.50e-04 & 1.00 & 1.77e-01 & 0.35 & 4930\\
\hline    
\end{tabular}                                  
\label{dgcalc:table:p8.3errors}                     
\end{table}

\medskip
{\bf Test 4 ($d=2$, $p>2$).}
Let $p=2.5$, $d=2$, $\W=(0,1)^2$.  Choose $F,g$ such that the exact solution is $u(x,y) = e^{x+y}$.  For this test we choose $\ge\equiv 100$.  Table \ref{dgcalc:table:d2p2.5g100errors} shows the errors and rates in the $L^p$ and $W^{1,p}$-norm for $u-u_h$, where $u_h\in V_h$ is the discrete minimizer  of \eqref{dgcalc:eqn:ritzmini}. Again for problems with smooth solutions and lack degeneracy in the interior, the table indicates that the DR method is converging to the correct solution and we have an optimal order convergence rates in the $W^{1,p}$ semi-norm with a sub-optimal convergence rate in the $L^p$-norm.

\begin{table}[h]  
\caption{The $L^p$ and $W^{1,p}$ errors and rates of convergence in $h$ for the discontinuous Ritz method \eqref{dgcalc:eqn:ritzmini} applied to $\mJ^{p}$ from \eqref{dgcalc:eqn:plapenergy} where $d=2$, $p=2.5$, and $\ge\equiv 100$.}                              
\centering                                     
\begin{tabular}{|c|c|c|c|c|}    
\hline                                         
$1/h$ & $\|u-u_h\|_{L^p(\W)}$ & rate & $\|\grad u-\grad_h u_h\|_{L^p(\W)}$ & rate\\              
\hline                                    
4 & 2.01e-02 & - & 2.79e-01 & - \\ 
\hline                                    
8 & 1.04e-02 & 0.94 & 1.33e-01 & 1.07 \\ 
\hline                                    
16 & 5.32e-03 & 0.98 & 6.36e-02 & 1.07 \\ 
\hline                                    
32 & 2.68e-03 & 0.99 & 3.06e-02 & 1.06 \\ 
\hline                                       
\end{tabular}                                  
\label{dgcalc:table:d2p2.5g100errors}                     
\end{table}


\begin{thebibliography}{10}
	
	\bibitem{Sp:BS}
	S.~C. Brenner and L.~R. Scott.
	\newblock {\em The Mathematical Theory of Finite Element Methods}, volume~15 of
	{\em Texts in Applied Mathematics}.
	\newblock Springer, New York, third edition, 2008.
	
	\bibitem{Buffa2009}
	A.~Buffa and C.~Ortner.
	\newblock Compact embeddings of broken {S}obolev spaces and applications.
	\newblock {\em IMA journal of numerical analysis}, 29(4):827--855, 2009.
	
	\bibitem{Cockburn1998}
	B.~Cockburn and C.-W. Shu.
	\newblock The local discontinuous {G}alerkin method for time-dependent
	convection-diffusion systems.
	\newblock {\em SIAM Journal on Numerical Analysis}, 35(6):2440--2463, 1998.
	
	\bibitem{Dacorogna2007}
	B.~Dacorogna.
	\newblock {\em Direct Methods in the Calculus of Variations}, volume~78.
	\newblock Springer Science \& Business Media, 2007.
	
	\bibitem{Dacorogna2015b}
	B.~Dacorogna.
	\newblock {\em Introduction to the Calculus of Variations}.
	\newblock Imperial College Press, London, third edition, 2015.
	
	\bibitem{Evans2010}
	L.~C. Evans.
	\newblock {\em Partial Differential Equations}, volume~19 of {\em Graduate
		Studies in Mathematics}.
	\newblock American Mathematical Society, Providence, RI, second edition, 2010.
	
	\bibitem{Feng2013}
	X.~Feng, T.~Lewis, and M.~Neilan.
	\newblock Discontinuous {G}alerkin finite element differential calculus and
	applications to numerical solutions of linear and nonlinear partial
	differential equations.
	\newblock {\em Journal of Computational and Applied Mathematics}, 299:68--91,
	2016.
	
	\bibitem{FengLi2015}
	X.~Feng and Y.~Li.
	\newblock Analysis of symmtric interior penalty discontinuous {G}alerkin
	methods the {A}llen-{C}ahn equation and its sharp interface limit the mean
	curvature flow.
	\newblock {\em IMA Journal on Numerical Analysis}, 35:1622--1651, 2015.
	
	\bibitem{Furihata2010}
	D.~Furihata and T.~Matsuo.
	\newblock {\em Discrete Variational Derivative Method: A Structure-preserving
		Numerical Method for Partial Differential Equations}.
	\newblock CRC Press, 2010.
	
	\bibitem{Lewis2014}
	T.~Lewis and M.~Neilan.
	\newblock Convergence analysis of a symmetric dual-wind discontinuous
	{G}alerkin method.
	\newblock {\em Journal of Scientific Computing}, 59(3):602--625, 2014.
	
	\bibitem{Nochetto2017}
	R.~H. Nochetto, S.~W. Walker, and W.~Zhang.
	\newblock A finite element method for nematic liquid crystals with variable
	degree of orientation.
	\newblock {\em SIAM Journal on Numerical Analysis}, 55(3):1357--1386, 2017.
	
	\bibitem{DGCalccode}
	S.~Schnake.
	\newblock A {M}atlab toolbox for the discontinuous {G}alerkin finite element
	numerical calculus, 2014.
	\newblock downloadable at
	https://bitbucket.org/stefanschnake/dgfenumericcalculus.
	
	\bibitem{Schnake2017}
	S.~Schnake.
	\newblock {\em Numerical Methods for Non-divergence Form Second Order Elliptic
		Partial Differential Equations and Discontinuous Ritz Methods for Problems
		from the Calculus of Variations}.
	\newblock PhD thesis, The University of Tennessee, 2017.
	
	\bibitem{Stamm} Stamm, B., Wihler, T.P.: A total variation discontinuous Galerkin approach for image restoration. Int. J. Numer. Anal. Model. 12(1), 81–93 (2015)
	
\end{thebibliography}
\bibliographystyle{abbrv}

\end{document}